\documentclass{amsart}
\usepackage{amssymb,amsmath,stmaryrd,mathrsfs}

\def\definetac{\newif\iftac}    % Can't define a \newif inside another \if!
\ifx\tactrue\undefined
  \definetac
  \ifx\state\undefined\tacfalse\else\tactrue\fi
\fi

\def\definebeamer{\newif\ifbeamer}
\ifx\beamertrue\undefined
  \definebeamer
  \ifx\uncover\undefined\beamerfalse\else\beamertrue\fi
\fi

\def\definecref{\newif\ifcref}
\ifx\creftrue\undefined
  \definecref
  \creffalse
\fi

\iftac\else\usepackage{amsthm}\fi
\usepackage{tikz,tikz-cd}
\usetikzlibrary{arrows}
\ifbeamer\else
  \usepackage{enumitem}
  \usepackage{xcolor}
  \definecolor{darkgreen}{rgb}{0,0.45,0} 
  \iftac\else\usepackage[pagebackref,colorlinks,citecolor=darkgreen,linkcolor=darkgreen]{hyperref}\fi
  \renewcommand*{\backref}[1]{}
  \renewcommand*{\backrefalt}[4]{({%
      \ifcase #1 Not cited.%
            \or On p.~#2%
            \else On pp.~#2%
      \fi%
    })}
\fi
\usepackage{mathtools}          % for all sorts of things
\usepackage{graphics}           % for \scalebox, used in \widecheck
\usepackage{ifmtarg}            % used in \jd
\usepackage{microtype}
\usepackage{braket}             % for \Set, etc.

\usepackage{url}                % for citations to web sites
\usepackage{xspace}             % put spaces after a \command in text
\ifcref\usepackage{cleveref,aliascnt}\fi
\usepackage[status=draft,author=]{fixme}
\fxusetheme{color}
\usepackage{mathpartir}

\makeatletter
\let\ea\expandafter

\def\mdef#1#2{\ea\ea\ea\gdef\ea\ea\noexpand#1\ea{\ea\ensuremath\ea{#2}\xspace}}
\def\alwaysmath#1{\ea\ea\ea\global\ea\ea\ea\let\ea\ea\csname your@#1\endcsname\csname #1\endcsname
  \ea\def\csname #1\endcsname{\ensuremath{\csname your@#1\endcsname}\xspace}}

\DeclareRobustCommand\widecheck[1]{{\mathpalette\@widecheck{#1}}}
\def\@widecheck#1#2{%
    \setbox\z@\hbox{\m@th$#1#2$}%
    \setbox\tw@\hbox{\m@th$#1%
       \widehat{%
          \vrule\@width\z@\@height\ht\z@
          \vrule\@height\z@\@width\wd\z@}$}%
    \dp\tw@-\ht\z@
    \@tempdima\ht\z@ \advance\@tempdima2\ht\tw@ \divide\@tempdima\thr@@
    \setbox\tw@\hbox{%
       \raise\@tempdima\hbox{\scalebox{1}[-1]{\lower\@tempdima\box
\tw@}}}%
    {\ooalign{\box\tw@ \cr \box\z@}}}

\newcount\foreachcount

\def\foreachletter#1#2#3{\foreachcount=#1
  \ea\loop\ea\ea\ea#3\@alph\foreachcount
  \advance\foreachcount by 1
  \ifnum\foreachcount<#2\repeat}

\def\foreachLetter#1#2#3{\foreachcount=#1
  \ea\loop\ea\ea\ea#3\@Alph\foreachcount
  \advance\foreachcount by 1
  \ifnum\foreachcount<#2\repeat}

\def\definescr#1{\ea\gdef\csname s#1\endcsname{\ensuremath{\mathscr{#1}}\xspace}}
\foreachLetter{1}{27}{\definescr}
\def\definecal#1{\ea\gdef\csname c#1\endcsname{\ensuremath{\mathcal{#1}}\xspace}}
\foreachLetter{1}{27}{\definecal}
\def\definebold#1{\ea\gdef\csname b#1\endcsname{\ensuremath{\mathbf{#1}}\xspace}}
\foreachLetter{1}{27}{\definebold}
\def\definebb#1{\ea\gdef\csname d#1\endcsname{\ensuremath{\mathbb{#1}}\xspace}}
\foreachLetter{1}{27}{\definebb}
\def\definefrak#1{\ea\gdef\csname f#1\endcsname{\ensuremath{\mathfrak{#1}}\xspace}}
\foreachletter{1}{9}{\definefrak}
\foreachletter{10}{27}{\definefrak}
\foreachLetter{1}{27}{\definefrak}
\def\definesf#1{\ea\gdef\csname i#1\endcsname{\ensuremath{\mathsf{#1}}\xspace}}
\foreachletter{1}{6}{\definesf}
\foreachletter{7}{14}{\definesf}
\foreachletter{15}{27}{\definesf}
\foreachLetter{1}{27}{\definesf}
\def\definebar#1{\ea\gdef\csname #1bar\endcsname{\ensuremath{\overline{#1}}\xspace}}
\foreachLetter{1}{27}{\definebar}
\foreachletter{1}{8}{\definebar} % \hbar is something else!
\foreachletter{9}{15}{\definebar} % \obar is something else!
\foreachletter{16}{27}{\definebar}
\def\definetil#1{\ea\gdef\csname #1til\endcsname{\ensuremath{\widetilde{#1}}\xspace}}
\foreachLetter{1}{27}{\definetil}
\foreachletter{1}{27}{\definetil}
\def\definehat#1{\ea\gdef\csname #1hat\endcsname{\ensuremath{\widehat{#1}}\xspace}}
\foreachLetter{1}{27}{\definehat}
\foreachletter{1}{27}{\definehat}
\def\definechk#1{\ea\gdef\csname #1chk\endcsname{\ensuremath{\widecheck{#1}}\xspace}}
\foreachLetter{1}{27}{\definechk}
\foreachletter{1}{27}{\definechk}
\def\defineul#1{\ea\gdef\csname u#1\endcsname{\ensuremath{\underline{#1}}\xspace}}
\foreachLetter{1}{27}{\defineul}
\foreachletter{1}{27}{\defineul}

\def\autofmt@n#1\autofmt@end{\mathrm{#1}}
\def\autofmt@b#1\autofmt@end{\mathbf{#1}}
\def\autofmt@d#1#2\autofmt@end{\mathbb{#1}\mathsf{#2}}
\def\autofmt@c#1#2\autofmt@end{\mathcal{#1}\mathit{#2}}
\def\autofmt@s#1#2\autofmt@end{\mathscr{#1}\mathit{#2}}
\def\autofmt@f#1\autofmt@end{\mathsf{#1}}
\def\autofmt@k#1\autofmt@end{\mathfrak{#1}}
\def\autofmt@u#1\autofmt@end{\underline{\smash{\mathsf{#1}}}}
\def\autofmt@U#1\autofmt@end{\underline{\underline{\smash{\mathsf{#1}}}}}
\def\autofmt@h#1\autofmt@end{\widehat{#1}}
\def\autofmt@r#1\autofmt@end{\overline{#1}}
\def\autofmt@t#1\autofmt@end{\widetilde{#1}}
\def\autofmt@k#1\autofmt@end{\check{#1}}

\def\auto@drop#1{}
\def\autodef#1{\ea\ea\ea\@autodef\ea\ea\ea#1\ea\auto@drop\string#1\autodef@end}
\def\@autodef#1#2#3\autodef@end{%
  \ea\def\ea#1\ea{\ea\ensuremath\ea{\csname autofmt@#2\endcsname#3\autofmt@end}\xspace}}
\def\autodefs@end{blarg!}
\def\autodefs#1{\@autodefs#1\autodefs@end}
\def\@autodefs#1{\ifx#1\autodefs@end%
  \def\autodefs@next{}%
  \else%
  \def\autodefs@next{\autodef#1\@autodefs}%
  \fi\autodefs@next}

\DeclareSymbolFont{bbold}{U}{bbold}{m}{n}
\DeclareSymbolFontAlphabet{\mathbbb}{bbold}

\mdef\delbar{\overline{\partial}}

\mdef\hf{\textstyle\frac12 }
\mdef\thrd{\textstyle\frac13 }
\mdef\qtr{\textstyle\frac14 }

\mdef\Id{\mathrm{Id}}
\mdef\id{\mathrm{id}}
\alwaysmath{ell}
\alwaysmath{infty}

\alwaysmath{odot}
\def\frc#1/#2.{\frac{#1}{#2}}   % \frc x^2+1 / x^2-1 .
\mdef\ten{\mathrel{\otimes}}

\mdef\sqten{\mathrel{\boxtimes}}
\DeclareFontFamily{U}{mathb}{\hyphenchar\font45}
\DeclareFontShape{U}{mathb}{m}{n}{
      <5> <6> <7> <8> <9> <10> gen * mathb
      <10.95> mathb10 <12> <14.4> <17.28> <20.74> <24.88> mathb12
      }{}
\DeclareSymbolFont{mathb}{U}{mathb}{m}{n}
\DeclareFontSubstitution{U}{mathb}{m}{n}
\DeclareMathSymbol{\dotplus}       {2}{mathb}{"00}% name to be checked
\DeclareMathSymbol{\dotdiv}        {2}{mathb}{"01}% name to be checked
\DeclareMathSymbol{\dottimes}      {2}{mathb}{"02}% name to be checked
\DeclareMathSymbol{\divdot}        {2}{mathb}{"03}% name to be checked
\DeclareMathSymbol{\udot}          {2}{mathb}{"04}% name to be checked
\DeclareMathSymbol{\square}        {2}{mathb}{"05}% name to be checked
\DeclareMathSymbol{\Asterisk}      {2}{mathb}{"06}
\DeclareMathSymbol{\bigast}        {1}{mathb}{"06}
\DeclareMathSymbol{\coAsterisk}    {2}{mathb}{"07}
\DeclareMathSymbol{\bigcoast}      {1}{mathb}{"07}
\DeclareMathSymbol{\circplus}      {2}{mathb}{"08}% name to be checked
\DeclareMathSymbol{\pluscirc}      {2}{mathb}{"09}% name to be checked
\DeclareMathSymbol{\convolution}   {2}{mathb}{"0A}% name to be checked
\DeclareMathSymbol{\divideontimes} {2}{mathb}{"0B}% name to be checked
\DeclareMathSymbol{\blackdiamond}  {2}{mathb}{"0C}% name to be checked
\DeclareMathSymbol{\sqbullet}      {2}{mathb}{"0D}% name to be checked
\DeclareMathSymbol{\bigstar}       {2}{mathb}{"0E}
\DeclareMathSymbol{\bigvarstar}    {2}{mathb}{"0F}
\DeclareMathSymbol{\corresponds}   {3}{mathb}{"1D}% name to be checked
\DeclareMathSymbol{\updownarrows}          {3}{mathb}{"D6}
\DeclareMathSymbol{\downuparrows}          {3}{mathb}{"D7}
\DeclareMathSymbol{\Lsh}                   {3}{mathb}{"E8}
\DeclareMathSymbol{\Rsh}                   {3}{mathb}{"E9}
\DeclareMathSymbol{\dlsh}                  {3}{mathb}{"EA}
\DeclareMathSymbol{\drsh}                  {3}{mathb}{"EB}
\DeclareMathSymbol{\looparrowdownleft}     {3}{mathb}{"EE}
\DeclareMathSymbol{\looparrowdownright}    {3}{mathb}{"EF}
\DeclareMathSymbol{\curvearrowleftright}   {3}{mathb}{"F2}
\DeclareMathSymbol{\curvearrowbotleft}     {3}{mathb}{"F3}
\DeclareMathSymbol{\curvearrowbotright}    {3}{mathb}{"F4}
\DeclareMathSymbol{\curvearrowbotleftright}{3}{mathb}{"F5}
\DeclareMathSymbol{\leftsquigarrow}        {3}{mathb}{"F8}
\DeclareMathSymbol{\rightsquigarrow}       {3}{mathb}{"F9}
\DeclareMathSymbol{\leftrightsquigarrow}   {3}{mathb}{"FA}
\DeclareMathSymbol{\lefttorightarrow}      {3}{mathb}{"FC}
\DeclareMathSymbol{\righttoleftarrow}      {3}{mathb}{"FD}
\DeclareMathSymbol{\uptodownarrow}         {3}{mathb}{"FE}
\DeclareMathSymbol{\downtouparrow}         {3}{mathb}{"FF}
\DeclareMathSymbol{\varhash}       {0}{mathb}{"23}

\DeclareMathOperator\colim{colim}

\mdef\we{\overset{\sim}{\longrightarrow}}
\mdef\leftwe{\overset{\sim}{\longleftarrow}}

\let\fib\twoheadrightarrow

\def\rightarrowtailfill@{\arrowfill@{\Yright\joinrel\relbar}\relbar\rightarrow}
\newcommand\xrightarrowtail[2][]{\ext@arrow 0055{\rightarrowtailfill@}{#1}{#2}}

\def\twoheadrightarrowfill@{\arrowfill@{\relbar\joinrel\relbar}\relbar\twoheadrightarrow}
\newcommand\xtwoheadrightarrow[2][]{\ext@arrow 0055{\twoheadrightarrowfill@}{#1}{#2}}

\def\slashedarrowfill@#1#2#3#4#5{%
  $\m@th\thickmuskip0mu\medmuskip\thickmuskip\thinmuskip\thickmuskip
   \relax#5#1\mkern-7mu%
   \cleaders\hbox{$#5\mkern-2mu#2\mkern-2mu$}\hfill
   \mathclap{#3}\mathclap{#2}%
   \cleaders\hbox{$#5\mkern-2mu#2\mkern-2mu$}\hfill
   \mkern-7mu#4$%
}
\def\rightslashedarrowfill@{%
  \slashedarrowfill@\relbar\relbar\mapstochar\rightarrow}
\newcommand\xslashedrightarrow[2][]{%
  \ext@arrow 0055{\rightslashedarrowfill@}{#1}{#2}}
\mdef\hto{\xslashedrightarrow{}}
\mdef\htoo{\xslashedrightarrow{\quad}}

\def\jd#1{\@jd#1\ej}
\def\@jd#1|-#2\ej{\@@jd#1,,\;\vdash\;\left(#2\right)}
\def\@@jd#1,{\@ifmtarg{#1}{\let\next=\relax}{\left(#1\right)\let\next=\@@@jd}\next}
\def\@@@jd#1,{\@ifmtarg{#1}{\let\next=\relax}{,\,\left(#1\right)\let\next=\@@@jd}\next}
\def\jdm#1{\@jdm#1\ej}
\def\@jdm#1|-#2\ej{\@@jd#1,,\\\vdash\;\left(#2\right)}
\long\def\my@drawfill#1#2;{%
\@skipfalse
\fill[#1,draw=none] #2;
\@skiptrue
\draw[#1,fill=none] #2;
}
\newif\if@skip
\newcommand{\skipit}[1]{\if@skip\else#1\fi}
\newcommand{\drawfill}[1][]{\my@drawfill{#1}}

\newcounter{nodemaker}
\newif\ifhyperref
\@ifpackageloaded{hyperref}{\hyperreftrue}{\hyperreffalse}
\iftac
  \let\your@state\state
  \def\state#1{\my@state#1}
  \def\my@state#1.{\gdef\currthmtype{#1}\your@state{#1.}}
  \let\your@staterm\staterm
  \def\staterm#1{\my@staterm#1}
  \def\my@staterm#1.{\gdef\currthmtype{#1}\your@staterm{#1.}}
  \let\@defthm\newtheorem
  \def\switchtotheoremrm{\let\@defthm\newtheoremrm}
  \def\defthm#1#2#3{\@defthm{#1}{#2}} % Ignore the third argument (for cleveref only)
  \let\your@section\section
  \def\section{\gdef\currthmtype{section}\your@section}
  \def\currthmtype{}
  \ifhyperref
    \def\autoref#1{\ref*{label@name@#1}~\ref{#1}}
  \else
    \def\autoref#1{\ref{label@name@#1}~\ref{#1}}
  \fi
  \AtBeginDocument{%
    \let\old@label\label%
    \def\label#1{%
      {\let\your@currentlabel\@currentlabel%
        \edef\@currentlabel{\currthmtype}%
        \old@label{label@name@#1}}%
      \old@label{#1}}
  }
  \let\cref\autoref
\else\ifcref
  \def\defthm#1#2#3{%
    \newaliascnt{#1}{thm}
    \newtheorem{#1}[#1]{#2}
    \aliascntresetthe{#1}
    \crefname{#1}{#2}{#3}% following brace must be on separate line to support poorman cleveref sed file
  }
\else
  \ifhyperref
    \def\defthm#1#2#3{% Ignore the third argument (for cleveref only)
      \newtheorem{#1}{#2}[section]%
      \expandafter\def\csname #1autorefname\endcsname{#2}%
      \expandafter\let\csname c@#1\endcsname\c@thm}
  \else
    \def\defthm#1#2#3{\newtheorem{#1}[thm]{#2}} % Ignore the third argument (for cleveref only)
    \ifx\SK@label\undefined\let\SK@label\label\fi
    \let\old@label\label
    \let\your@thm\@thm
    \def\@thm#1#2#3{\gdef\currthmtype{#3}\your@thm{#1}{#2}{#3}}
    \def\currthmtype{}
    \def\label#1{{\let\your@currentlabel\@currentlabel\def\@currentlabel%
        {\currthmtype~\your@currentlabel}%
        \SK@label{#1@}}\old@label{#1}}
    \def\autoref#1{\ref{#1@}}
  \fi
  \let\cref\autoref
\fi\fi

\newtheorem{thm}{Theorem}[section]
\ifcref
  \crefname{thm}{Theorem}{Theorems}
\else
  
\fi
\defthm{cor}{Corollary}{Corollaries}
\defthm{prop}{Proposition}{Propositions}
\defthm{lem}{Lemma}{Lemmas}
\defthm{sch}{Scholium}{Scholia}
\defthm{assume}{Assumption}{Assumptions}
\defthm{claim}{Claim}{Claims}
\defthm{conj}{Conjecture}{Conjectures}
\defthm{hyp}{Hypothesis}{Hypotheses}
\iftac\switchtotheoremrm\else\theoremstyle{definition}\fi
\defthm{defn}{Definition}{Definitions}
\defthm{notn}{Notation}{Notations}
\iftac\switchtotheoremrm\else\theoremstyle{remark}\fi
\defthm{rmk}{Remark}{Remarks}
\defthm{eg}{Example}{Examples}
\defthm{egs}{Examples}{Examples}
\defthm{ex}{Exercise}{Exercises}
\defthm{ceg}{Counterexample}{Counterexamples}

\ifcref
  \crefformat{section}{\S#2#1#3}
  \Crefformat{section}{Section~#2#1#3}
  \crefrangeformat{section}{\S\S#3#1#4--#5#2#6}
  \Crefrangeformat{section}{Sections~#3#1#4--#5#2#6}
  \crefmultiformat{section}{\S\S#2#1#3}{ and~#2#1#3}{, #2#1#3}{ and~#2#1#3}
  \Crefmultiformat{section}{Sections~#2#1#3}{ and~#2#1#3}{, #2#1#3}{ and~#2#1#3}
  \crefrangemultiformat{section}{\S\S#3#1#4--#5#2#6}{ and~#3#1#4--#5#2#6}{, #3#1#4--#5#2#6}{ and~#3#1#4--#5#2#6}
  \Crefrangemultiformat{section}{Sections~#3#1#4--#5#2#6}{ and~#3#1#4--#5#2#6}{, #3#1#4--#5#2#6}{ and~#3#1#4--#5#2#6}
  \crefformat{appendix}{Appendix~#2#1#3}
  \Crefformat{appendix}{Appendix~#2#1#3}
  \crefrangeformat{appendix}{Appendices~#3#1#4--#5#2#6}
  \Crefrangeformat{appendix}{Appendices~#3#1#4--#5#2#6}
  \crefmultiformat{appendix}{Appendices~#2#1#3}{ and~#2#1#3}{, #2#1#3}{ and~#2#1#3}
  \Crefmultiformat{appendix}{Appendices~#2#1#3}{ and~#2#1#3}{, #2#1#3}{ and~#2#1#3}
  \crefrangemultiformat{appendix}{Appendices~#3#1#4--#5#2#6}{ and~#3#1#4--#5#2#6}{, #3#1#4--#5#2#6}{ and~#3#1#4--#5#2#6}
  \Crefrangemultiformat{appendix}{Appendices~#3#1#4--#5#2#6}{ and~#3#1#4--#5#2#6}{, #3#1#4--#5#2#6}{ and~#3#1#4--#5#2#6}
  \crefformat{subappendix}{\S#2#1#3}
  \Crefformat{subappendix}{Section~#2#1#3}
  \crefrangeformat{subappendix}{\S\S#3#1#4--#5#2#6}
  \Crefrangeformat{subappendix}{Sections~#3#1#4--#5#2#6}
  \crefmultiformat{subappendix}{\S\S#2#1#3}{ and~#2#1#3}{, #2#1#3}{ and~#2#1#3}
  \Crefmultiformat{subappendix}{Sections~#2#1#3}{ and~#2#1#3}{, #2#1#3}{ and~#2#1#3}
  \crefrangemultiformat{subappendix}{\S\S#3#1#4--#5#2#6}{ and~#3#1#4--#5#2#6}{, #3#1#4--#5#2#6}{ and~#3#1#4--#5#2#6}
  \Crefrangemultiformat{subappendix}{Sections~#3#1#4--#5#2#6}{ and~#3#1#4--#5#2#6}{, #3#1#4--#5#2#6}{ and~#3#1#4--#5#2#6}
  \crefname{part}{Part}{Parts}
  \crefname{figure}{Figure}{Figures}
\fi

\iftac
  
  \let\your@endproof\endproof
  \def\my@endproof{\your@endproof}
  \def\endproof{\my@endproof\gdef\my@endproof{\your@endproof}}
  \def\qedhere{\tag*{\endproofbox}\gdef\my@endproof{\relax}}
\fi

\iftac
  \def\pr@@f[#1]{\subsubsection*{\sc #1.}}
\fi

\def\thmqedhere{\expandafter\csname\csname @currenvir\endcsname @qed\endcsname}

\ifbeamer\else

\fi

\ifbeamer\else
  \setitemize[1]{leftmargin=2em}
  \setenumerate[1]{leftmargin=*}
\fi

\iftac
  \let\c@equation\c@subsection
\else
  \let\c@equation\c@thm
\fi
\numberwithin{equation}{section}

\ifcref\else
  \@ifpackageloaded{mathtools}{\mathtoolsset{showonlyrefs,showmanualtags}}{}
\fi

\alwaysmath{alpha}
\alwaysmath{beta}
\alwaysmath{gamma}
\alwaysmath{Gamma}
\alwaysmath{delta}
\alwaysmath{Delta}
\alwaysmath{epsilon}
\mdef\ep{\varepsilon}
\alwaysmath{zeta}
\alwaysmath{eta}
\alwaysmath{theta}
\alwaysmath{Theta}
\alwaysmath{iota}
\alwaysmath{kappa}
\alwaysmath{lambda}
\alwaysmath{Lambda}
\alwaysmath{mu}
\alwaysmath{nu}
\alwaysmath{xi}
\alwaysmath{pi}
\alwaysmath{rho}
\alwaysmath{sigma}
\alwaysmath{Sigma}
\alwaysmath{tau}
\alwaysmath{upsilon}
\alwaysmath{Upsilon}
\alwaysmath{phi}
\alwaysmath{Pi}
\alwaysmath{Phi}
\mdef\ph{\varphi}
\alwaysmath{chi}
\alwaysmath{psi}
\alwaysmath{Psi}
\alwaysmath{omega}
\alwaysmath{Omega}
\let\al\alpha
\let\be\beta
\let\gm\gamma

\let\de\delta

\makeatother

\usepackage{hyperref}
\usepackage[all]{xy}
\usepackage{graphicx}
\title{On a model invariance problem in homotopy type theory}
\author{Anthony Bordg}
\thanks{This material is based upon work supported by grant GA CR P201/12/G028}

\newcommand{\fsix}[6]{\begin{array}{rcl} #1&\longrightarrow &#2\\ #3&\longmapsto &#4\\ #5&\longmapsto &#6\\\\ \end{array}} 

\newcommand{\pushoutcorner}[1][dr]{\save*!/#1+1.2pc/#1:(1,-1)@^{|-}\restore}
\newdir{ >}{{}*!/-5pt/\dir{>}}
\newcommand{\pullbackcorner}[1][dr]{\save*!/#1-1.2pc/#1:(-1,1)@^{|-}\restore}

\def\Gpd{\mathbf{Gpd}}
\def\GGpd{\mathbf{Gpd}^{\mathbb{Z}_2}}

\def\Z2{\mathbb{Z}_2}
\def\f{_{\mathbf{f}}}
\def\0{\mathbf{0}}
\def\1{\mathbf{1}}
\def\cI{\mathbf{\check{I}}}
\def\tr{\bigtriangledown}

\begin{document}

\begin{abstract}
	In this article the author endows the functor category $[\mathbf{B}(\mathbb{Z}_2),\Gpd]$ with the structure of a type-theoretic fibration category with a \emph{univalent universe} using the so-called injective model structure. It gives us a new model of Martin-L\"of type theory with dependent sums, dependent products, identity types and a univalent universe. This model, together with the model (developed by the author in an other work) in the same underlying category together with the very same universe that turned out to be provably not univalent with respect to projective fibrations, provides an example of two Quillen equivalent model categories that host different models of type theory. Thus, we provide a counterexample to the \emph{model invariance problem} formulated by Michael Shulman. 
\end{abstract}
\maketitle

\section{Introduction}
\label{sec:introduction}
	
This article is a contribution to the ongoing effort to find models of the \emph{Univalent Foundations} \cite{bordg:uf} introduced by Vladimir Voevodsky. In particular, Shulman \cite{shulman:invdia, shulman:elreedy, shulman:eiuniv} with his notion of \emph{type-theoretic fibration categories} prompted the development of models of the \emph{Univalence Axiom} in functor categories. In \cite{shulman:invdia} Shulman endowed the functor category $[\mathcal{D}, \mathscr{C}]$, where $\mathcal{D}$ is an \emph{inverse category} and $\mathscr{C}$ is a type-theoretic fibration category with a univalent universe, with the structure of a type-theoretic fibration category with a univalent universe by using the so-called \emph{Reedy model structure}. In \cite{shulman:elreedy} Shulman endowed the category $[\mathcal{D}, \mathbf{sSet}]$, where $\mathcal{D}$ is any \emph{elegant Reedy category} and $\mathbf{sSet}$ is the category of simplicial sets, with the structure of a type-theoretic fibration category with a univalent universe, again by using the Reedy model structure. The reader should note that inverse categories are particular cases of elegant Reedy categories that are themselves particular cases of (strict) \emph{Reedy categories}. Since Reedy categories do not allow non-trivial isomorphism, this kind of index categories has strong limitations. Moreover, it is useful to note that the class of elegant Reedy categories is precisely the class of index categories for which the Reedy model structure and the so-called \emph{injective model structure} on a functor category are the same. Thus, despite technical challenges that might be difficult to overcome the injective model structure on a functor category seems a reasonable candidate to find models of the Univalence Axiom in functor categories. However, in \cite{shulman:eiuniv} Shulman used a different model structure to give models in EI-diagrams. An EI-category is a category where every endomorphism is an isomorphism, groups are particular interesting cases. According to Shulman ``[He] constructed a model in a certain model
category that presents the homotopy theory of presheaves on an
EI-category, but it is not the injective model structure on a functor
category. In the case of [the target category $\Gpd$ of groupoids and the index category] $\mathbb{Z}/2\mathbb{Z}$ or any other group $G$, [his] model specializes to the slice category $\Gpd/\mathbf{B}G$, which is well-known to be Quillen equivalent to, but not identical to, the injective model structure on $\Gpd^G$'' (private communication). Before Shulman's work in \cite{shulman:eiuniv} the author had worked out in his PhD thesis \cite[Ch.5]{bordg:thesis} the details of a type-theoretic fibration category  with a univalent universe using the injective model structure on $[\mathbf{B}(\mathbb{Z}_2), \Gpd]$, overcoming the technical challenge of the presence of a non-trivial automorphism in the index category at least in the simple but important case of the target category $\Gpd$ with its univalent universe of sets (discrete groupoids). \\
Moreover, since it is well known that the projective and injective model categories are Quillen equivalent, our new univalent universe, together with our proof \cite{bordg:inadequacy} that this same universe is not univalent with respect to the projective structure on the same underlying category, provides a counterexample to the \emph{model invariance problem}\footnote{\url{https://ncatlab.org/homotopytypetheory/show/model+invariance+problem}} formulated by Michael Shulman (as part of a list of open problems\footnote{\url{https://ncatlab.org/homotopytypetheory/show/open+problems}} in \emph{Homotopy Type Theory}): ``Show that the interpretation of type theory is independent of the model category chosen to present an (infinity,1)-category. Of course, the details depend on the chosen type theory.''. In the present work, we consider a type theory with a unit type, dependent sums, dependent products, intensional identity types and a universe type, and by an interpretation of that type theory we mean precisely Shulman's notion of a type-theoretic fibration category with a universe which is intended to give an interpretation of such a type theory. 

\subsection*{Acknowledgments}

I would like to thank Andr\'e Hirschowitz, Peter LeFanu Lumsdaine, and Michael Shulman for helpful discussions.

\section{The injective type-theoretic fibration structure on $\GGpd$}
\label{sec:ttfc}

We denote the functor category $[\mathbf{B}(\mathbb{Z}_2), \Gpd]$ simply by $\GGpd$. The reader should note that an object in $\GGpd$ is nothing but a groupoid $A$ equipped with an involution $\al\colon A\rightarrow A$, \textit{i.e.} an automorphism satisfying the equation $\al\circ \al = id$. A morphism $f\colon A\rightarrow B$ in $\GGpd$ is nothing but an equivariant functor, \textit{i.e.} $f$ satisfies $f\circ\al = \be\circ f$. \\
We recall the definitions of \emph{type-theoretic fibration category} \cite[Def.7.1]{shulman:eiuniv}, \emph{type-theoretic model category} \cite[Def.2.12]{shulman:invdia} and their link \cite[Prop.2.13]{shulman:invdia}.

\begin{defn}
	\label{def:ttfc}
	A \textbf{type-theoretic fibration category} is a category $\mathscr{C}$ with :
	\begin{enumerate}[leftmargin=*,label=(\arabic*)]
		\item A terminal object 1.\label{item:cat1}
		\item A subcategory of \textbf{fibrations} containing all the isomorphisms and all the morphisms with codomain 1. A morphism is called an \textbf{acyclic cofibration} if it has the left lifting property with respect to all fibrations.\label{item:cat2}
		\item All pullbacks of fibrations exist and are fibrations.\label{item:cat3}
		\item The pullback functor along any fibration has a right adjoint that preserves fibrations. \label{item:cat4}
		\item Every morphism factors as an acyclic cofibration followed by a fibration.\label{item:cat5}
	\end{enumerate}
\end{defn}

\begin{rmk}
	This category-theoretic structure corresponds to an interpretation into a category of a type theory with a unit type, dependent sums, dependent products, and intensional identity types.
\end{rmk}	

\begin{defn}
\label{def:ttmc}
	A \textbf{type-theoretic model category} is a model category $\mathscr{M}$ satisfying the following additional properties.
	\begin{enumerate}[leftmargin=*,label=(\roman*)]
		\item The pullback functor along a fibration preserves acyclic cofibrations. \label{item:ttmc1}
		\item The pullback functor $g^*$ along a fibration $g$ has a right adjoint $\Pi_g$. \label{item:ttmc2}
	\end{enumerate}
\end{defn}

\begin{prop}
\label{prop:ttmc-ttfc}
	If $\mathscr{M}$ is a type-theoretic model category, then its full subcategory $\mathscr{M}\f$ of fibrant objects is a type-theoretic fibration category.
\end{prop}
\begin{proof}
\label{proof:prop:ttmc-ttfc}
	Since 1 is a fibrant object, \ref{item:cat1} is satisfied. The wide subcategory of $\mathscr{M}\f$ with fibrations as morphisms contains all the isomorphisms between fibrant objects and by definition all the morphisms with codomain 1, so \ref{item:cat2} holds. The fiber product associated with two fibrations between fibrant objects is fibrant, hence pullbacks of fibrations still exist and are fibrations. So, \ref{item:cat3} is satisfied. Since the middleman in the factorization of any morphism between fibrant objects by an acyclic cofibration followed by a fibration is a fibrant object, \ref{item:cat5} is still true. Last, by adjunction $\Pi_g$ preserves fibrations if and only if $g^*$ preserves acyclic cofibrations, so by \ref{item:ttmc1} we have \ref{item:cat4}.  
\end{proof}

The goal of this section consists in proving that $\GGpd$ equipped with the so-called \emph{injective model structure} is a type-theoretic model category, hence $(\GGpd)\f$ is a type-theoretic fibration category. \\
We denote by $\1$ the terminal object in the category $\Gpd$ of groupoids, and by a slight abuse of notation $\1$ will also denote the terminal object of $\GGpd$, namely the groupoid $\1$ together with the identity involution. The letter $\bI$ will denote the groupoid with two distinct objects and one isomorphism $\phi\colon 0\rightarrow 1$ between them. Recall that $\Gpd$ has a canonical model structure where the weak equivalences are the equivalences of groupoids. The fibrations are the functors with isomorphism-lifting, namely the functors with the right lifting property with respect to the inclusion $i\colon \1\hookrightarrow\bI$. The cofibrations are the injective-on-objects functors. \\
Given a \emph{combinatorial model category} $\mathscr{M}$ and a small category $\mathcal{I}$, there exists the injective model structure on $[\mathcal{I},\mathscr{M}]$ (see \cite[A.3.3]{lurie:higher-topoi} for details). In this case a morphism $f\in [\mathcal{I},\mathscr{M}]$ is a weak equivalence (\textit{resp.} a cofibration) if $f$ is an objectwise weak equivalence (\textit{resp.} an objectwise cofibration). A morphism is a fibration if it has the right lifting property with respect to every acyclic cofibration (\textit{i.e.} a morphism which is simultaneously a weak equivalence and a cofibration). Since $\Gpd$ together with its canonical model structure is combinatorial, there exists the injective model structure on $\GGpd$. Given a morphism $f$ in $\GGpd$, $\underline{f}$ will denote its image under the forgetful functor that maps an equivariant functor to its underlying functor between groupoids.

\begin{prop}
\label{prop:rightadjoint}
	Let $\mathscr{C}$ be a category together with a distinguished class of morphisms called fibrations and $\mathcal{I}$ a small category. Moreover, assume that $\mathscr{C}$ is locally presentable and for every fibration $g$ the pullback functor along $g$ exists and has a right adjoint. Then for any objectwise fibration $g$ in $[\mathcal{I},\mathscr{C}]$ the pullback functor along $g$ has a right adjoint.
\end{prop}
\begin{proof}
\label{proof:prop:rightadjoint}
	Let $g\colon A\rightarrow B$ be an objectwise fibration in $[\mathcal{I},\mathscr{C}]$. Since $\mathscr{C}$ is locally presentable, so are the slices $[\mathcal{I},\mathscr{C}]/B$ and $[\mathcal{I},\mathscr{C}]/A$. Hence, $g^*\colon [\mathcal{I},\mathscr{C}]/B\rightarrow[\mathcal{I},\mathscr{C}]/A$ has a right adjoint if and only if it preserves all small colimits. So, let $\mathcal{D}$ be any small category and $F\colon \mathcal{D}\rightarrow [\mathcal{I},\mathscr{C}]/B$ any  functor such that $\colim F$ exists. One has to provide an isomorphism of the form
	$$g^*(\colim F) \cong \colim(g^*\circ F)$$. Knowing the nature of colimits in a slice category, it is enough to check that $\text{dom}(g^*(\colim F))$ is isomorphic to $\colim (\text{dom} \circ g^* \circ F)$. Since colimits in a functor category are pointwise, it is enough to check that $\text{dom}(g^*(\colim F))(x)$ is isomorphic to $\underset{d}{\colim}\,[\text{dom}(g^*(F(d)))(x)]$ for every $x\in\mathcal{I}$. But pullbacks are pointwise in a functor category, so we have an isomorphism between 
	$\text{dom}(g^*(\colim F))(x)$ and $\text{dom}((g_x)^*\,(\colim F)_x)$. Since $g_x$ is a fibration, by assumption $(g_x)^*$ has a right adjoint and so it preserves all small colimits. Moreover, one has an isomorphism of the form $(\colim F)_x\cong 
	\underset{d}{\colim} \,F(d)_x$. As a consequence, one has the following sequences of isomorphisms
	\begin{align*}
	(g_x)^*((\colim F)_x) & \cong (g_x)^*(\underset{d}{\colim}\,F(d)_x) \\
	 & \cong\underset{d}{\colim}\,[(g_x)^* (F(d)_x)] \\
	\end{align*}
	, and finally 
	\begin{align*}
	\text{dom}((g_x)^*((\colim F)_x)) & \cong \text{dom}(\underset{d}{\colim}\,[(g_x)^* (F(d)_x)]) \\
	& \cong \underset{d}{\colim}\,[\text{dom}((g_x)^* (F(d)_x))] \\
	& \cong \underset{d}{\colim}\,[\text{dom}(g^*\,F(d)) (x)]
	\end{align*}.
\end{proof}

\begin{prop}
\label{prop:ttmc}
	Let $\mathscr{M}$ be a type-theoretic model category whose underlying model category is combinatorial and $\mathcal{I}$ a small category. The category $[\mathcal{I},\mathscr{M}]$ together with the injective model structure is a type-theoretic model category.
\end{prop}
\begin{proof}
\label{proof:prop:ttmc}
	Since a combinatorial model category is locally presentable as a category and fibrations with respect to the injective model structure are in particular objectwise fibrations, by \ref{prop:rightadjoint} the pullback functor along a fibration has a right adjoint. Moreover, since pullbacks are pointwise in a functor category and acyclic cofibrations are objectwise with respect to the injective model structure, we conclude by \ref{item:ttmc1} for $\mathscr{M}$.	
\end{proof}

Since $\Gpd$ together with its canonical model structure is a type-theoretic model category \cite[Examples 2.16]{shulman:invdia} and is combinatorial as a model category, by \ref{prop:ttmc} we conclude that $\GGpd$ together with the injective model structure is a type-theoretic model category. So, by \ref{prop:ttmc-ttfc} $(\GGpd)\f$ is a type-theoretic fibration category. If we wish to take this model of type theory further, we need a better control on the fibrations of the injective model structure. This is the topic of the next section. 

\section{The injective model structure on $\GGpd$ made explicit}
\label{sec:injstruc}

\begin{notn}
\label{notn:3.1}
	We denote by $\cI$ the groupoid $\bI$ together with the involution that maps $\phi$ to $\phi^{-1}$. Also, we will denote by $\tr$ the groupoid that extends $\cI$ and its involution by a third point, fixed under the $\Z2$-action, denoted $2$ and a second non-identity isomorphism $\psi\colon 1\rightarrow 2$ whose image by the involution is $\psi\circ\phi$ (there are only two non-identity isomorphisms in $\tr$ and their composition). The morphism $i'\colon\cI\hookrightarrow \tr$ is the corresponding inclusion.
\end{notn}

\begin{rmk}
\label{rmk:}
	Note that $i'$ is an objectwise acyclic cofibration, hence it is an acyclic cofibration with respect to the injective model structure. Moreover, note that the morphism $\cI\rightarrow\1$ is an objectwise fibration. However, consider the following lifting problem, $$\xymatrix{\cI \ar@{=}[r] \ar@{^{(}->}[d] & \cI \ar[d] \\ 
		\tr \ar[r] & \1}$$.
	A diagonal filler cannot exist, since the fixed point $2$ of $\tr$ should be mapped to a fixed point, but such a fixed point does not exist in $\cI$. So, there exist objectwise fibrations that are not fibrations with respect to the injective model structure. 	
\end{rmk}

\begin{notn}
\label{notn:3.3}
	Remember that $\lbrace i\rbrace$, where $i$ is the inclusion $\1\hookrightarrow \bI$, is a set of generating acyclic cofibration with respect to the canonical model structure on $\Gpd$. Moreover, the forgetful functor from $\GGpd$ to $\Gpd$ has a left adjoint $S$ that maps a groupoid $A$ to $A\textstyle\coprod A$ together with the involution that swaps the two copies of $A$.
\end{notn}

\begin{prop}
\label{prop:genacycliccof}
	Let $f$ be a morphism in $\GGpd$, the following are equivalent :
	\begin{enumerate}[label=(\roman*)]
		\item $f$ is an acyclic cofibration with respect to the injective model structure.
		\item $f$ is a transfinite composition of pushouts of elements of the set $\lbrace S(i),i' \rbrace$.
	\end{enumerate}
\end{prop}
\begin{proof}
\label{proof:prop:genacycliccof}
	The implication $(ii)\Rightarrow(i)$ is clear, since $S(i)$ and $i'$ are objectwise acyclic cofibrations, so they are acyclic cofibrations for the injective model structure. Moreover, the class of acyclic cofibrations is closed under pushouts and transfinite compositions. \\
	Conversely, let $f\colon A\rightarrow B$ be an acyclic cofibration. Since $f$ is an objectwise acyclic cofibration, $\underline{f}$ is (isomorphic to) the inclusion of a full subgroupoid of $B$ which is equivalent to $B$. Let $((\text{Ob}B\setminus\text{Ob}A)/\Z2, \leqslant)$ be the set of orbits of $\text{Ob}B\setminus\text{Ob}A$ under the $\Z2$-action together with a well-ordering, let $\lambda$ be the order type of this well-ordered set, and let $g\colon (\text{Ob}B\setminus\text{Ob}A)/\Z2\rightarrow \lambda$ be an order-preserving bijection. By transfinite recursion we define a $\lambda$-sequence $X$, where we add the elements of $\text{Ob}B\setminus\text{Ob}A$ to $A$ by following our well-ordering. Take $X_0\coloneqq A$. For $\gm$ such that $\gm + 1 <\lambda$, let $s$ be the element of $(\text{Ob}B\setminus\text{Ob}A)/\Z2$ that corresponds to $\gm + 1$ under the bijection $g$. We built $X_{\gm + 1}$ as follows. If $s$ is a singleton with unique element $x$, \textit{i.e.} $x$ is fixed under the $\Z2$-action, then $f$ being essentially surjective there exists an isomorphism $\varphi\colon y\rightarrow x$ with $y\in A$. In this case $X_{\gm + 1}$ is the following pushout
	$$\xymatrix{\cI \ar[r]\ar@{^{(}->}[d]_{i'} & X_{\gm} \ar[d] \\
		\tr \ar[r] & X_{\gm + 1}\pushoutcorner}$$,
	where the upper horizontal morphism maps $\phi$ to $\be(\varphi)^{-1}\circ \varphi$. Otherwise, the orbit $s$ is $\lbrace x, \be(x)\rbrace$, and we define $X_{\gm + 1}$ as the following pushout
	$$\xymatrix{S(\1) \ar[r]\ar@{^{(}->}[d]_{S(i)} & X_{\gm} \ar[d] \\
		S(\bI)\ar[r] & X_{\gm + 1} \pushoutcorner}$$, where the upper horizontal arrow maps $0$ to $y$ (and $1$ to $\al(y)$). Last, if $\gm$ is a limit ordinal, then $X_{\gm}$ is $\underset{\delta<\gm}{\colim}\, X_{\delta}$. For every $\gm<\lambda$, $X_{\gm}$ is a full subgroupoid of $B$ stable under the involution $\be$ on $B$, and $f$ is the transfinite composition of the $\lambda$-sequence $X$.
\end{proof}

\begin{prop}
\label{prop:injfib}
	Let $f$ be a morphism in $\GGpd$, the following are equivalent :
	\begin{enumerate}[label=(\roman*)]
		\item $f$ is a fibration with respect to the injective model structure on $\GGpd$.
		\item $f$ has the right lifting property with respect to the elements of the set $\lbrace S(i), i'\rbrace$.
	\end{enumerate}
\end{prop}
\begin{proof}
\label{proof:prop:injfib}
Straightforward with \ref{prop:genacycliccof} and the fact that if a class of maps has the right lifting property with respect to a set $J$ of maps, then it has the right lifting property with respect to the relative J-cell complexes (\textit{i.e} the transfinite compositions of pushouts of elements of $J$) \cite[Prop.10.5.10]{hirschhorn:modelcats}.
\end{proof}

\section{A univalent universe in the type-theoretic fibration category $(\GGpd)\f$}
\label{sec:univalentuniv}

We recall the notion of a \emph{universe} \cite[Definition 6.12]{shulman:invdia} in a type-theoretic fibration category.

\begin{defn}
	\label{def:universe}
	A fibration $p: \widetilde{U}\fib U$ in a type-theoretic fibration category $\mathscr{C}$ is a \textbf{universe} if the following hold.
	\begin{enumerate}
		\item Pullbacks of $p$ are closed under composition and contain the identities.\label{item:u1}
		\item If $f: B\fib A$ and $g: A\fib C$ are pullbacks of $p$, so is $\Pi_g f \fib C$.\label{item:u2}
		\item If $A\fib C$ and $B\fib C$ are pullbacks of $p$, then any morphism $f: A\to B$ over $C$ factors as an acyclic cofibration followed by a pullback of $p$.\label{item:u3}
	\end{enumerate}
\end{defn}

\begin{defn}
	\label{defn:smallfib}
	Given a universe $p\colon\widetilde{U}\rightarrow U$ in a type-theoretic fibration category, a \textbf{small fibration}, or a $U$-small fibration, is a pullback of $p$. 	
\end{defn}

\begin{rmk}
	\label{rmk:modelttwithuniv}
	A universe in a type-theoretic fibration category interprets a universe type in type theory.	
\end{rmk}

Also, we recall what it means for a universe in a type-theoretic fibration category to be \emph{univalent}\label{univalenceproperty} (see also \cite[section 7]{shulman:invdia}). Let $\mathsf{Type}$ be a universe in the type theory under consideration. Given two \emph{small} types, \textit{i.e.} two elements of $\mathsf{Type}$, there is the type of weak equivalences between them. In a type-theoretic fibration category with a universe, this dependent type is represented by a fibration $E \fib U\times U$. Moreover, there is a natural map $U \rightarrow E$ that sends a type to its identity equivalence. By \ref{item:cat5} one can factor the diagonal map $\de\colon U \rightarrow U\times U$ as an acyclic cofibration followed by a fibration in the following commutative diagram,
$$\xymatrix{U \ar[r]\ar@{ >->}[d]^{\rotatebox[origin=c]{90}{$\sim$}} & E \ar@{->>}[d] \\ PU \ar@{->>}[r] \ar@{-->}[ru] & U\times U}$$.
The universe $p\colon \widetilde{U} \fib U$ is univalent if the map $U \rightarrow E$ is a right homotopy equivalence, or equivalently (by the \emph{2-out-of-3} property and the fact that $U$ is fibrant like any object of a type-theoretic fibration category) if the dashed map is a right homotopy equivalence.

Given $\kappa$ an inaccessible cardinal, we recall \cite{bordg:inadequacy} below the construction of a \emph{non-univalent universe} in the type-theoretic fibration category $\GGpd$ together with the so-called \emph{projective model structure}.
\begin{itemize}
	\item The objects of the groupoid $\widetilde{U}$ are dependent tuples of the form $(A, B, \varphi, a)$, where $A, B$ are $\kappa$-small discrete groupoids, $\varphi\colon A \rightarrow B$ is an isomorphism in $\Gpd$, and $a$ is an object of $A$.
	\item The morphisms in $\widetilde{U}$ between $(A, B, \varphi, a)$ and $(C, D, \psi, c)$ are pairs of the form $(\rho\colon A\rightarrow C, \tau\colon B\rightarrow D)$ such that $\psi\circ \rho = \tau\circ \varphi$ and $\rho (a) = c$. 
\end{itemize}
The composition in $\widetilde{U}$ is given by 
$$(\rho',\tau')\circ(\rho,\tau)\coloneqq(\rho'\circ\rho,\tau'\circ\tau)$$. Note that $\widetilde{U}$ is a groupoid. Indeed, the inverse of the morphism $(\rho,\tau)$ is given by
$$(\rho,\tau)^{-1}\coloneqq(\rho^{-1},\tau^{-1})$$. We equip $\widetilde{U}$ with the involution $\tilde{\upsilon}$ as follows,
$$\fsix{\tilde{\upsilon}\colon\widetilde{U}}{\widetilde{U}}{(A,B,\varphi,a)}{(B,A,\varphi^{-1},\varphi(a))}{(\rho,\tau)}{(\tau,\rho)}$$. One denotes by $U$ the ``unpointed'' version of $\widetilde{U}$, \textit{i.e.} objects are of the form $(A,B,\varphi)$ and morphisms of the form $(\rho,\tau)$, with its corresponding involution $\upsilon$. We define the morphism $p$ in $\GGpd$ as the projection 
$$\fsix{p\colon\widetilde{U}}{U}{(A,B,\varphi,a)}{(A,B,\varphi)}{(\rho,\tau)}{(\rho,\tau)}$$.

First, we will prove that the morphism $p\colon \widetilde{U}\rightarrow U$ is a fibration between fibrant objects \emph{with respect to the injective model structure} on $\GGpd$.

\begin{lem}
\label{lem:4.4}
	The morphism $p\colon\widetilde{U} \rightarrow U$ is a fibration with respect to the injective model structure. 
\end{lem}
\begin{proof}
\label{proof:lem:4.4}
	Thanks to \ref{prop:injfib} it suffices to prove that $p$ has the right lifting property with respect to $S(i)$ and $i'$. We already know that $p$ has the right lifting property against $S(i)$. Indeed, $p$ is a fibration with respect to the projective model structure, \textit{i.e.} an objectwise fibration, and $S(i)$ is a generating acyclic cofibration with respect to that model structure (see \cite{bordg:inadequacy} for details, in particular the beginning of section 2 and lemma 4.5).\\
	Assume we have a lifting problem as follows
	$$\xymatrix{\cI \ar@{^{(}->}[d]_{i'}\ar[r]^f & \widetilde{U} \ar[d]^{p} \\
		\tr \ar[r]_{g} & U}$$.
	Let $f(\phi)$ be the pair $(\rho,\tau)\colon (A,B,\varphi,a)\rightarrow (B,A,\varphi^{-1},\varphi(a))$.
	Since $f$ is equivariant, we have
	$f(\phi^{-1}) = \widetilde{\upsilon}(f(\phi))$, and one concludes that $\tau = \rho^{-1}$ and $\tau(\varphi(a)) = a$. By commutativity of the diagram one has $p(f(\phi)) = g(\phi)$, hence $\text{dom}(g(\psi)) = (B,A,\varphi^{-1})$.
	Let $g(\psi)\colon (B,A,\varphi^{-1})\rightarrow (C,D,\eta)$ be the pair $(\sigma,\chi)$.
	Since $g$ is equivariant, $g(\psi\circ\phi)$ is equal to $\upsilon(g(\psi))$.
	So, one has the equality $\sigma\circ\rho = \chi$. Moreover, note that $g(2)$ is a fixed point of $U$, hence $D = C$ and $\eta$ is an involution. Now, we define a diagonal filler $j$ as follows. Take $j(\phi)= f(\phi)$, $j(2) = (C,C,\eta,\chi(a))$, and $j(\psi) = (\sigma,\chi)$ seen as a morphism in $\widetilde{U}$ from $(B,A,\varphi^{-1},\varphi(a))$ to $(C,C,\eta,\chi(a))$ (indeed, $\varphi(a) = \rho(a)$, hence $\sigma(\varphi(a)) = \sigma(\rho(a)) = \chi(a)$).
\end{proof}

\begin{lem}
\label{lem:4.5}
	The groupoids $\widetilde{U}$ and $U$ together with their involutions $\widetilde{\upsilon}$ and $\upsilon$ are fibrant objects of $\GGpd$ with respect to the injective model structure.
\end{lem}
\begin{proof}
\label{proof:lem:4.5}
	We start with $U$. It suffices by \ref{prop:injfib} to prove that the unique morphism from $U$ to $\1$ has the right lifting property with respect to $S(i)$ and $i'$. 
	First, assume that we have the following lifting problem
	$$\xymatrix{\cI \ar@{^{(}->}[d]_{i'}\ar[r]^f & U \ar[d] \\ 
		\tr \ar[r] & \1}$$.
	Let $f(\phi)$ be the pair $(\rho,\tau)\colon (A,B,\varphi)\rightarrow (B,A,\varphi^{-1})$.
	Since $f$ is equivariant, one concludes $(\rho^{-1},\tau^{-1}) = (\tau,\rho)$. Hence, one has $\tau = \rho^{-1}$. We define a morphism $j\colon \tr\rightarrow U$ by $j(\phi) = f(\phi)$, $j(2) = (A,A,\tau\circ\varphi)$, and $j(\psi) = (\varphi^{-1},\tau\circ \varphi)$. The reader can easily check that $j$ is a diagonal filler.\\
	Second, $U$ is a projective fibrant object and $S(i)$ is a projective acylic cofibration, so $U\rightarrow \1$ has the right lifting property with respect to $S(i)$. \\
	Next, recall from \ref{lem:4.4} that $p$ is a fibration and fibrations are closed under composition. Thus, we deduce the fibrancy of $\widetilde{U}$ from the fibrancy of $U$ in the following commutative diagram
	$$\xymatrix{\widetilde{U} \ar@{-->}[rd] \ar@{->>}[d]_{p} & \\
		U \ar@{->>}[r] & \1}$$.
\end{proof}

\begin{thm}
\label{thm:universe}
	The morphism $p\colon\widetilde{U}\rightarrow U$ is a universe in the type-theoretic fibration structure on $(\GGpd)\f$ given in section 2.
\end{thm}
\begin{proof}
\label{proof:thm:universe}
	It follows from \ref{lem:4.4} and \ref{lem:4.5} that $p$ is a fibration in $(\GGpd)\f$ with respect to the injective model structure. Since small fibrations, \textit{i.e.} pullbacks of $p$, and right adjointness, when they exist, are categorical notions they are the same in the projective and injective settings, hence conditions \ref{item:u1} and \ref{item:u2} follow from their counterparts in \cite[Thm.4.11]{bordg:inadequacy}. Moreover, since projective acyclic cofibrations are in particular objectwise acyclic cofibrations, \ref{item:u3} follows from its counterpart \textit{ibid.} .   
\end{proof}

In the rest of this article we will prove that $p$ is a univalent universe. The first step consists in constructing specific path objects in $\GGpd$ with respect to the injective model structure. Let $f\colon A\rightarrow C$ be a morphism in $\GGpd$. By the universal property of the pullback, one gets the diagonal morphism $\delta$ as follows
$$\xymatrix{A \ar@{-->}[rd]^{\delta}\ar@/^2pc/[rrrd]^{\text{id}}\ar@/_2pc/[rddd]_{\text{id}} & & & \\ &{A\times_C A}\pullbackcorner \ar[rr]\ar[dd] & & A\ar[dd]^f \\ \\ & A \ar[rr]_f & & C}$$.
We define a groupoid $P_CA$ together with an involution $\pi_CA$ as follows. The objects of $P_CA$ are tuples $(x,y,\varphi)$, where $\varphi\colon x\rightarrow y$ is an isomorphism in $A$ such that $f(\varphi)$ is the identity morphism. A morphism in $P_CA$ between $(x,y,\varphi)$ and $(x',y',\varphi')$ is a pair $(\rho,\tau)$, where $\rho\colon x\rightarrow x'$ and $\tau\colon y\rightarrow y'$ are isomorphisms in $A$ such that $f(\rho) = f(\tau)$ and $\varphi'\circ\rho = \tau\circ\varphi$. The composition in $P_CA$ is given componentwise, \textit{i.e.} $(\rho',\tau')\circ (\rho,\tau)$ is $(\rho'\circ\rho,\tau'\circ\tau)$ whenever it makes sense. The inverse of $(\rho,\tau)$ is $(\rho^{-1},\tau^{-1})$. Define the involution $\pi_CA$ as follows 
$$\fsix{\pi_CA\colon P_CA}{P_CA}{(x,y,\varphi)}{(\alpha(x),\alpha(y),\alpha(\varphi))}{(\rho,\tau)}{(\alpha(\rho),\alpha(\tau))}$$
, where $\alpha$ is the involution on $A$. We define $\delta_1$, $\delta_2$ in $\GGpd$ as the following morphisms 
$$\fsix{\delta_1\colon A}{P_CA}{x}{(x,x,1_x)}{\varphi}{(\varphi,\varphi)}$$
$$\fsix{\delta_2\colon P_CA}{A\times_C A}{(x,y,\varphi)}{(x,y)}{(\rho,\tau)}{(\rho,\tau)}$$.
The morphisms $\delta_1$ and $\delta_2$ are equivariant and $\delta = \delta_2\circ\delta_1$.

\begin{prop}
\label{prop:4.7}
	Given $f\colon A\rightarrow C$ in $\GGpd$, $P_CA$ is a very good path object with respect to the injective fibration structure on $\GGpd$.
\end{prop}
\begin{proof}
\label{proof:prop:4.7}
	We have to prove that $\delta_1$ is an acyclic cofibration and $\delta_2$ is a fibration \textit{wrt.} the injective model structure. We start with $\delta_1$. It suffices to prove that $\underline{\delta_1}$ is an acyclic cofibration of groupoids. Clearly, it is an injective-on-objects functor. Moreover, $\delta_1$ is essentially surjective. Indeed, let $(x,y,\varphi)$ be an element of $P_CA$, then $(\varphi^{-1},1_y)$ is an isomorphism in $P_CA$ between $\delta_1(y) = (y,y,1_y)$ and $(x,y,\varphi)$.
	It remains to prove that $\delta_1$ is a fully faithful functor. But, for every morphism $(\rho,\tau)\colon (x,x,1_x)\rightarrow (y,y,1_y)$ in $P_CA$, one has $\rho = \tau$. We conclude that the induced map $A(x,y)\rightarrow P_CA(\delta_1(x),\delta_1(y))$ is a bijection for every pair $(x,y)\in A^2$.\\
	Now, we prove that $\delta_2$ is a fibration. First, the morphism $\delta_2$ has the right lifting property \textit{wrt.} $S(i)$. Indeed, $\delta_2$ is a projective fibration, in other words $\underline{\delta_2}$ has the isomorphism-lifting property. Let $(\rho,\tau)\colon (x,y)\rightarrow (x',y')$ be an isomorphism in $A\times_C~A$, $\varphi\colon x\rightarrow y$ an isomorphism such that $f(\varphi)$ is the identity, then $(\rho,\tau)\colon (x,y,\varphi)\rightarrow (x',y',\tau\circ\varphi\circ\rho^{-1})$ is an isomorphism in $P_CA$ above $(\rho,\tau)$. \\
	Second, consider the following lifting problem 
	$$\xymatrix{\cI \ar@{^{(}->}[d]_{i'}\ar[r]^g & P_CA \ar[d]^{\delta_2} \\ \tr \ar[r]_-h & A\times_C A}$$.
	We define a diagonal filler $j$ as follows.
	Take $j(\phi) = g(\phi)$. Let us assume that $g(\phi)$ is $(\rho,\tau)\colon (x,y,\varphi)\rightarrow (\alpha(x),\alpha(y),\alpha(\varphi))$.
	Since $g$ is equivariant, note that $(\rho^{-1},\tau^{-1}) = (\alpha(\rho),\alpha(\tau))$ and $\alpha(\varphi)\circ\rho = \tau \circ\varphi$.
	Moreover, let $h(2)$ be $(x',y')$ and $h(\psi)$ be $(\rho',\tau')\colon (\alpha(x),\alpha(y))\rightarrow (x',y')$. 
	The reader can easily check that $(\rho',\tau')$ is an isomorphism in $P_CA$ from $(\al(x),\al(y),\al(\varphi))$ to $(x',y',\tau'\circ\al(\varphi)\circ \rho'^{-1})$, with this last point being fixed under the involution $\pi_CA$ (indeed, since $h$ is equivariant, note that $\al(\rho') = \rho'\circ\rho$ and $\al(\tau') = \tau'\circ\tau$). We take $j(\psi)\coloneqq (\rho',\tau')$. \\
	So, by \ref{prop:injfib} $\delta_2$ is a fibration.
\end{proof}

\begin{prop}
\label{prop:pathfibrancy}
	If $f\colon A\rightarrow C$ is a fibration and $A$ is fibrant with respect to the injective model structure on $\GGpd$, then $P_CA$ is a fibrant object.
\end{prop}
\begin{proof}
\label{proof:prop:pathfibrancy}
	First, since the class of fibrations is closed under pullbacks and $f$ is a fibration, the first projection $pr_1\colon A\times_C A\to A$ is a fibration. Then, using the facts that $A$ is fibrant, the class of fibrations is closed under compositions, and the fact that $\delta_2$ is a fibration as a result of \ref{prop:4.7}, we conclude by considering the following commutative diagram
	$$\xymatrix{P_CA \ar@{->>}[d]_{\delta_2}\ar@{-->}[rdd] & \\
		A\times_C A \ar@{->>}[d]_{pr_1} & \\
		A \ar@{->>}[r] & \1}$$.
\end{proof}

\begin{rmk}
\label{rmk:4.9}
	The proposition \ref{prop:pathfibrancy} proves, under the assumption that $f\colon A\rightarrow C$ is a fibration in $(\GGpd)\f$, that $P_CA$ is really a path object for our type-theoretic fibration structure on $(\GGpd)\f$.  
\end{rmk}

In the case where $C\coloneqq \1$, $A\coloneqq U$, and $f\colon A\rightarrow C$ is the unique morphism from $U$ to $\1$, $U$ being fibrant $P_1U$ lives in $(\GGpd)\f$ as noted in \ref{rmk:4.9}, and tracing through the interpretation of type theory we find that the space E of equivalences over $U\times U$ is isomorphic to $P_1U$.

\begin{cor}
\label{thm:univalence}
	In the type-theoretic fibration structure on $(\GGpd)\f$ given in section \ref{sec:ttfc}, the universe $p\colon \widetilde{U} \rightarrow U$ satisfies the univalence property. 
\end{cor}
\begin{proof}
\label{proof:thm:univalence}
	Recall (\textit{cf.} the beginning of section \ref{sec:univalentuniv}) that we have to prove that the upper horizontal arrow in the following commutative diagram
	$$\xymatrix{U \ar[r]\ar@{ >->}[d]^{\rotatebox[origin=c]{90}{$\sim$}} & E \ar@{->>}[d] \\ PU \ar@{->>}[r] & U\times U}$$,
	which maps a small type to its identity equivalence, is a right homotopy equivalence. But this morphism is isomorphic to $\delta_1$, so by \ref{prop:4.7} and \ref{lem:4.5} it is an acyclic cofibration with a fibrant domain, hence a right homotopy equivalence.
\end{proof}

\section{Conclusion}
\label{sec:conclusion}

Our new interpretation of the Univalent Foundations in the category $[\mathbf{B}(\mathbb{Z}_2), \Gpd]$ is an incremental progress in the direction of finding new type-theoretic fibration categories together with a universe satisfying the Univalence Axiom using the injective model structure on functor categories. This new model together with its previous twin model using the projective model structure on $[\mathbf{B}(\mathbb{Z}_2), \Gpd]$ provides a counterexample to Shulman's \emph{model invariance problem} by showing that two Quillen equivalent model categories can host different interpretations of type theory. So, a Quillen equivalence between model categories is not trivial in the context of type theory and can make a difference with respect to the interpretation of the type theory under consideration. Last, a conjecture to the effect that ``equivalent homotopy theories have equivalent internal type theories'' should not mention model categories but type-theoretic fibration categories and an adequate notion of equivalence thereof.

\bibliographystyle{alpha}
\bibliography{all}
	
\end{document}